\documentclass[leqno]{article}
\usepackage[frenchb,english]{babel}
\usepackage[utf8]{inputenc}
\usepackage{amsmath}
\usepackage{amssymb}
\usepackage{amsfonts}
\usepackage{enumerate}
\usepackage{vmargin}
\usepackage[all]{xy}
\usepackage{mathrsfs}
\usepackage{mathtools}
\usepackage{lmodern}
\usepackage{comment}
\usepackage[colorlinks=true,linkcolor=blue]{hyperref}
\setmarginsrb{3cm}{3cm}{3.5cm}{3cm}{0cm}{0cm}{1.5cm}{3cm}

\newcommand{\B}{\ensuremath{\mathrm{B}}}
\newcommand{\C}{\ensuremath{\mathbb{C}}}
\newcommand{\D}{\mathrm{D}}
\newcommand{\E}{\ensuremath{\mathbb{E}}}

\let\H\relax 
\newcommand{\H}{\mathrm{H}} 

\let\L\relax 
\newcommand{\L}{\mathrm{L}}
\newcommand{\M}{\mathrm{M}}
\newcommand{\N}{\ensuremath{\mathcal{N}}}

\newcommand{\R}{\ensuremath{\mathbb{R}}}

\newcommand{\W}{\mathrm{W}}

\newcommand{\e}{\mathrm{e}}

\renewcommand{\d}{\mathop{}\mathopen{}\mathrm{d}} 

\let\i\relax 
\newcommand{\i}{\mathrm{i}}
\newcommand{\ovl}{\overline}
\newcommand{\Id}{\mathrm{Id}} 
\newcommand{\VN}{\mathrm{VN}} 
\newcommand{\la}{\langle}\newcommand{\ra}{\rangle}
\renewcommand{\leq}{\ensuremath{\leqslant}}
\renewcommand{\geq}{\ensuremath{\geqslant}}
\newcommand{\qed}{\hfill \vrule height6pt  width6pt depth0pt}
\newcommand{\norm}[1]{ \| #1  \|}
\newcommand{\bnorm}[1]{ \big\| #1  \big\|}

\newcommand{\xra}{\xrightarrow} 
\newcommand{\co}{\colon}

\newcommand{\ot}{\otimes}

\newcommand{\BMO}{{\mathrm{BMO}}}
\newcommand{\scr}{\mathscr} 

\newcommand{\ov}{\overset}

\DeclareMathOperator{\Aut}{Aut} 

\DeclareMathOperator{\tr}{Tr}
\newtheorem{thm}{Theorem}[section]

\newtheorem{prop}[thm]{Proposition}

\newtheorem{lemma}[thm]{Lemma}

\newtheorem{remark}[thm]{Remark}

\newenvironment{proof}[1][]{\noindent {\it Proof #1} : }{\hbox{~}\qed
\smallskip
}

\numberwithin{equation}{section}

\usepackage{tocloft}
\usepackage[nottoc,notlot,notlof]{tocbibind}

\let\OLDthebibliography\thebibliography
\renewcommand\thebibliography[1]{
  \OLDthebibliography{#1}
  \setlength{\parskip}{0pt}
  \setlength{\itemsep}{0pt plus 0.3ex}
}

\allowdisplaybreaks

\begin{document}
\selectlanguage{english}
\title{\bfseries{Markov dilations of semigroups of Fourier multipliers}}
\date{}
\author{\bfseries{C\'edric Arhancet}}

\maketitle

\begin{abstract}
We describe a Markov dilation for any weak* continuous semigroup $(T_t)_{t \geq 0}$ of selfadjoint unital completely positive Fourier multipliers acting on the group von Neumann algebra $\VN(G)$ of a locally compact group $G$.
\end{abstract}


\makeatletter
 \renewcommand{\@makefntext}[1]{#1}
 \makeatother
 \footnotetext{
 2020 {\it Mathematics subject classification:}
 Primary 47A20, 47D03, 46L51.
\\
{\it Key words and phrases}: semigroups, dilations, Fourier multipliers, completely positive maps, crossed products.}

\tableofcontents

\section{Introduction}
\label{sec:Introduction}

Fourier multipliers are the most used operators in (noncommutative) harmonic analysis. In particular, the study of semigroups of these operators is a central issue. In the noncommutative setting, these are connected to a large number of topics, as approximation properties \cite{BrO08} \cite{Haa7879} or noncommutative geometry \cite{ArK22a}.

In this paper, we will focus on weak* continuous semigroup $(T_t)_{t \geq 0}$ of selfadjoint unital completely positive Fourier multipliers acting on the group von Neumann algebra $\VN(G)$ of an arbitrary locally compact group $G$. We construct a Markov dilation in the sense of \cite[p. 637]{JM10} for these semigroups, which allows us to state our main result Theorem \ref{Th-Markov-dilation-Fourier}.

This construction makes it possible to use martingale theory and obtain results in analysis with probability tools. Indeed, Junge and Mei studied in \cite{JM12} several $\BMO$-spaces associated to a Markov semigroup on a semifinite von Neumann algebra (i.e. a noncommutative $\L^\infty$-space). In particular, the authors obtained interpolation results. Their approach relies on Markov dilations of semigroups. Note that some extensions of these results were generalized to the $\sigma$-finite von Neumann algebras  in the paper \cite{Cas19}. Other applications of these dilations include some estimates  related to Riesz transforms \cite{JM12} and connected to the curvature assumption $\Gamma^2 \geq 0$, boundedness of the $\H^\infty$ functional calculus of the (negative) generator of the semigroup \cite{JiW17} \cite{FMS19} with applications to maximal inequalities and to ergodic theory.

Observe that Markov dilations are really different from the kind of dilation constructed in the papers \cite{Arh13}, \cite{Arh19} and \cite{Arh20} in the spirit of classical Fendler's isometric dilation theorem \cite{Fen97} (see also \cite{ALM14} \cite{AFM17} for related works). Furthermore, it is remarkable that the algebra associated to the dilation of this paper is the same that the algebra associated to the dilation constructed in \cite{Arh20}. However, we were unable to find a direct connection between the construction of this paper and the dilation described in \cite{Arh13}.

Note that Markov dilations of weak* continuous semigroups of selfadjoint unital completely positive measurable Schur multipliers on $\sigma$-finite measure spaces were explicitely constructed in \cite{Arh21}. In \cite{Cas19}, it is described how to obtain Markov dilations of radial semigroups on free Araki-Woods factors. The paper \cite{CJSZ20} contains a construction of Markov dilations of semigroups of double operator integrals. Finally, an unpublished paper \cite{JRS} of Junge, Ricard and Shlyakhtenko describes a construction of a Markov dilation for any weak* continuous semigroup of selfadjoint unital completely positive maps on a finite von Neumann algebra (see also the recent \cite{Wir22} for a related construction in the type III case). In the three preceding papers, the construction relies on the use of ultraproduct methods with sharp contrast with the ones of \cite{Arh21} and of this paper. The construction of concrete dilations and Markov dilations remains very important, especially for Ricci curvature bounds, see e.g. \cite{BGJ20}. 

In the discrete setting of a semigroup $(T^n)_{n \in \mathbb{N}}$ associated to an operator $T$ acting on a von Neumann algebra, a notion of Markov dilation was introduced in \cite{AnD06} and thoroughly investigating in \cite{HaM11}. Examples of such dilations are also provided in \cite{Ric08}.

\paragraph{Structure of the paper} The paper is organized as follows. The next section \ref{sec:preliminaries} gives background on probabilities, Fourier multipliers and crossed products. In Section \ref{sec-Markov}, we state and prove our dilation result and we describe a reversed result.


\section{Preliminaries}
\label{sec:preliminaries}

\paragraph{Isonormal processes} Let $H$ be a real Hilbert space. An $H$-isonormal process on a probability space $(\Omega,\mu)$ \cite[Definition 1.1.1]{Nua06} \cite[Definition 6.5]{Nee08} is a linear mapping $\W \co H \to \L^0(\Omega)$ from $H$ into the space $\L^0(\Omega)$ of measurable functions on $\Omega$ with the following properties:
\begin{flalign}
& \label{isonormal-gaussian} \text{for any $h \in H$ the random variable $\W(h)$ is a centered real Gaussian,} \\
&\label{esperance-isonormal} \text{for any } h_1, h_2 \in H \text{ we have } \E\big(\W(h_1) \W(h_2)\big)= \langle h_1, h_2\rangle_H, \\
&\label{density-isonormal} \text{the linear span of the products } \W(h_1)\W(h_2)\cdots \W(h_m), 
\text{ with } m \geq 0 \text{ and } h_1,\ldots, h_m \\
&\nonumber\text{in }H, \text{ is dense in the real Hilbert space $\L^2_\R(\Omega)$.}\end{flalign}  
Here we make the convention that the empty product, corresponding to $m=0$ in \eqref{density-isonormal}, is the constant function $1$. Moreover, $\E$ is used to denote expected value.

If $(e_i)_{i \in I}$ is an orthonormal basis of $H$ and if $(\gamma_i)_{i \in I}$ is a family of independent standard Gaussian random variables on a probability space $\Omega$ then for any $h \in H$, the family $(\gamma_i \langle h,e_i\rangle_H)_{i \in I}$ is summable in $\L^2(\Omega)$ and
\begin{equation}
\label{Concrete-W}
\W(h)
\ov{\mathrm{def}}{=} \sum_{i \in I} \gamma_i \langle h,e_i\rangle_H, \quad h \in H
\end{equation}
defines an $H$-isonormal process.

Recall that the span of elements $\e^{\i\W(h)}$ where $h \in H$ is weak* dense in $\L^\infty(\Omega)$ by \cite[Remark 2.15 p. 22]{Jan97}. It is easy to prove that we can replace $H$ by a dense subset of $H$ with Lemma \ref{Lemma-semigroup-continuous} below. 
Using \cite[Proposition E.2.2]{HvNVW18} with $t$ instead of $\xi$ and by observing by \eqref{esperance-isonormal} that the variance $\E(\W(h)^2)$ of the Gaussian variable $\W(h)$ is equal to $\norm{h}_H^2$, we see that
\begin{equation}
\label{Esperance-exponentielle-complexe}
\E\big(\e^{\i t\W(h)}\big)
=\e^{-\frac{t^2}{2} \norm{h}_H^2}, \quad t \in \R, h \in H.
\end{equation}

If $u \co H \to H$ is a contraction, we denote by $\Gamma^\infty(u) \co \L^\infty(\Omega) \to \L^\infty(\Omega)$ the (symmetric) second quantization of $u$ acting on the \textit{complex} Banach space $\L^\infty(\Omega)$. Recall that the map $\Gamma^\infty(u) \co \L^\infty(\Omega) \to \L^\infty(\Omega)$ preserves the integral\footnote{\thefootnote. That means that for any $f \in \L^\infty(\Omega)$ we have $\int_{\Omega} \Gamma^\infty(u)f \d\mu=\int_{\Omega} f \d\mu$.}. If $u$ is a surjective isometry we have
\begin{equation}
\label{SQ1}
\Gamma^\infty(u) \big(\e^{\i\W(h)}\big)
=\e^{\i\W(u(h))}, \quad h \in H
\end{equation}
and $\Gamma^\infty(u) \co \L^\infty(\Omega) \to \L^\infty(\Omega)$ is a $*$-automorphism of the von Neumann algebra $\L^\infty(\Omega)$. If $P \co H \to H$ is an ortogonal projection on a closed subspace $K$, the operator $\Gamma^\infty(u) \co \L^\infty(\Omega) \to \L^\infty(\Omega)$ is a faithful normal conditional expectation with range $\L^\infty(\Omega, \scr{F})$ where $\scr{F}$ is the $\sigma$-algebra generated by the random variables $\W(h)$ for $h \in K$, see \cite[Theorem 4.9 p.~46]{Jan97}. 

Furthermore, the second quantization functor $\Gamma$ satisfies the following elementary result \cite[Lemma 2.1]{Arh20}. In the first part, we suppose that the construction\footnote{\thefootnote. The existence of a proof of Lemma \ref{Lemma-semigroup-continuous} without \eqref{Concrete-W} is unclear.} is given by the concrete representation \eqref{Concrete-W}.   

\begin{lemma}
\label{Lemma-semigroup-continuous} 
\begin{enumerate}
	\item If $\L^\infty(\Omega)$ is equipped with the weak* topology then the map $H \to \L^\infty(\Omega)$, $h \mapsto \e^{\i\W(h)}$ is continuous.
	\item If $\pi \co G \to \B(H)$ is a strongly continuous orthogonal representation of a  locally compact group, then $G \to \B(\L^\infty(\Omega))$, $s \mapsto \Gamma^\infty(\pi_s)$ is a weak* continuous\footnote{\thefootnote. That means that $\B(\L^\infty(\Omega))$ is equipped with the point weak* topology.} representation on the Banach space $\L^\infty(\Omega)$.
\end{enumerate}
\end{lemma}

Let $H$ be a real Hilbert space. Following \cite[Definition 2.2]{NVW15} and \cite[Definition 6.11]{Nee08}, we say that an $\L^2_\R(\R^+,H)$-isonormal process $\W$ is an $H$-cylindrical Brownian motion. In this case, for any $t \geq 0$ and any $h \in H$, we let
\begin{equation}
\label{Def-Wt-h}
\W_t(h)\ov{\mathrm{def}}{=}\W\big(1_{[0,t]} \ot h\big).	
\end{equation}
We introduce the filtration $(\scr{F}_t)_{t \geq 0}$ defined by 
\begin{equation}
\label{Filtration-H-cylindrical}
\scr{F}_t\ov{\mathrm{def}}{=} \sigma\big(\W_r(h) : r \in [0,t], h \in H\big),
\end{equation} 
that is the $\sigma$-algebra generated by the random variables $\W_r(h)$ for $r \in [0,t]$ and $h \in H$.

By \cite[p.~77]{Nee08}, for any fixed $h \in H$, the family $(\W_t(h))_{t \geq 0}$ is a Brownian motion. This means by essentially \cite[Definition 6.2]{Nee08} that
\begin{flalign}
& \label{isonormal-almost} \text{$\W_0(h) = 0$ almost surely,} \\
&\label{difference-1} \text{$\W_t(h)-\W_u(h)$ is Gaussian with variance $(t-u)\norm{h}_H^2$ for any $0 \leq u \leq t$,} \\
&\label{difference-2} \text{$\W_t(h)-\W_u(h)$ is independent of $\{\W_r(h) :  r \in [0,u] \}$ for any $0 \leq u \leq t$}.  
\end{flalign}
Indeed by \cite[p.~163]{Nee08}, 
\begin{flalign}
\label{difference-3}
&\text{the increment $\W_t(h)-\W_u(h)$ is independent of the $\sigma$-algebra $\scr{F}_u$}.
\end{flalign}
Moreover, by \cite[p. 163]{Nee08} the family $(\W_t(h))_{t \geq 0}$ is a martingale with respect to $(\scr{F}_t)_{t \geq 0}$. In particular, the random variable $\W_t(h)$ is $\scr{F}_t$-measurable. If $0 \leq u \leq t$, note that 
$$
\norm{1_{]u,t]} \ot h}_{\L^2_\R(\R^+,H)}^2
= \bnorm{1_{]u,t]}}_{\L^2_\R(\R^+)}^2 \norm{h}_H^2
=(t-u) \norm{h}_H^2.
$$ 
Using \eqref{Esperance-exponentielle-complexe} together with the previous computation, we obtain
\begin{equation}
\label{Esperance-exponentielle-complexe-2}
\E\big(\e^{\i \W(1_{]u,t]} \ot h)}\big)
=\e^{-\frac{t-u}{2} \norm{h}_H^2}, \quad 0 \leq u \leq t,\ h \in H.
\end{equation}

\paragraph{Probabilities}
%
Let $\Omega$ be a probability space. If $f \in \L^1(\Omega)$ is independent of the sub-$\sigma$-algebra $\scr{F}$, then by \cite[Proposition 2.6.35]{HvNVW16} its conditional expectation $\E_\scr{F}(f)$ with respect to $\scr{F}$ is given by the constant function:
\begin{equation}
\label{Hy-2.6.35}
\E_\scr{F}(f) 
= \E(f).
\end{equation}
If $g \in \L^\infty(\Omega)$ is $\scr{F}$-measurable and $f \in \L^1(\Omega)$, we have by \cite[Proposition 2.6.31]{HvNVW16}
\begin{equation}
\label{Bimodule}
\E_\scr{F}(gf)
=g\,\E_\scr{F}(f).
\end{equation}
\paragraph{Group von Neumann algebras} 
Let $G$ be a locally compact group equipped with a fixed left Haar measure $\mu_G$. 
The group von Neumann algebra of $G$ is the von Neumann algebra generated by the set $\{\lambda_s : s \in G\}$ where 
$\lambda_s  \co \L^2(G) \to \L^2(G), f \mapsto (t \mapsto f(s^{-1}t))$ 
is the left translation by $s$.

\paragraph{Crossed products} 
We refer to \cite{Str81} and \cite{Tak03} for more information on crossed products. Let $M$ be a von Neumann algebra acting on a Hilbert space $H$. Let $G$ be a locally compact group equipped with some left Haar measure $\mu_G$. Let $\alpha \co G \to M$ be a representation of $G$ on $M$ which is weak* continuous, i.e. for any $x \in M$ and any $y \in M_*$, the map $G \to M$, $s \mapsto \langle \alpha_{s}(x), y \rangle_{M,M_*}$ is continuous. For any $x \in M$, we define the operators $\pi(x)\co \L^2(G,H) \to \L^2(G,H)$ \cite[(2) p. 263]{Str81} by
\begin{equation}
\label{}
\big(\pi(x) \xi \big)(s)
\ov{\mathrm{def}}{=}\alpha^{-1}_s(x) \xi(s),\quad
\quad  \xi \in \L^2(G, H), s \in G.
\end{equation}
These operators satisfy the following commutation relation \cite[(2) p. 292]{Str81}:
\begin{equation}
\label{commutation-rules}
(\lambda_s \ot \Id_H) \pi(x) (\lambda_s \ot \Id_H)^*
= \pi(\alpha_{s}(x)),
\quad x \in M, s \in G.
\end{equation}
Recall that the crossed product of $M$ and $G$ with respect to $\alpha$ is the von Neumann algebra 
$$
M \rtimes_\alpha G\ov{\mathrm{def}}{=} (\pi(M) \cup \{\lambda_s \ot \Id_{H}: s \in G\})''
$$ 
on the Hilbert space $\L^2(G,H)$ generated by the operators $\pi(x)$ and $\lambda_s \ot \Id_{H}$ where $x \in M$ and $s \in G$. By \cite[p. 263]{Str81}, $\pi$ is a normal injective $*$-homomorphism from $M$ into $M \rtimes_\alpha G$ (hence $\sigma$-strong* continuous). 

We denote by $\mathcal{K}(G,M)$ the space of $\sigma$-strong* continuous functions $f \co G \to M$, $s \mapsto f_s$ with compact support. If $f \in \mathcal{K}(G,M)$ then $f(G)$ is a $\sigma$-strong* compact subset of $M$, hence by \cite[Proposition 2.7 d)]{Osb14} a $\sigma$-strong* bounded subset of $M$. Hence it is a strong bounded subset and finally a norm-bounded subset of $M$ by the principle of uniform boundedness \cite[Theorem 1.8.9]{KaR97}. Note that by \cite[Proposition p. 186]{Str81} and \cite[p. 41]{Str81}, the bounded function $G \to M$, $s \mapsto \lambda_s \ot \Id_H$ is $\sigma$-strong* continuous and the norm-bounded function $s \mapsto \pi(f_s)$ is also $\sigma$-strong* continuous. Recall that the product of $M$ is $\sigma$-strong* continuous on bounded subsets by \cite[Proposition 2.4.5]{BrR87}. We infer\footnote{\thefootnote. In the book \cite{Str81}, the author considers weak* continuous functions, it is problematic since the product of $M$ is not weak* continuous even on bounded sets by \cite[Exercise 5.7.9]{KaR97} (indeed this latter fact is equivalent to the weak continuity of the product on bounded sets).} that the function $G \to M \rtimes_\alpha G$, $s \mapsto \pi(f_s)(\lambda_s \ot \Id_H)$ is $\sigma$-strong* continuous with compact support. 
So, by \cite[Lemma 2.2]{Arh20} and \cite[Corollary 2, III p. 38]{Bou04} we can define the element $\int_G f_s \rtimes \lambda_s \d\mu_G(s)$ of the crossed product $M \rtimes_\alpha G$ by 
\begin{equation}
\label{def-integrale-crossed}
\int_G f_s \rtimes \lambda_s \d\mu_G(s)
\ov{\mathrm{def}}{=} \int_G \pi(f_s)(\lambda_s \ot \Id_H) \d\mu_G(s).
\end{equation}

The following is a particular case\footnote{\thefootnote. The function $u \co G \to \mathrm{U}(M)$ is a $\alpha$-1-cocycle.} of \cite[Proposition 3.5]{Tak73} and its proof, see also \cite[Theorem 1.7 (ii) p. 241]{Tak03}. Note that the von Neumann algebra $M$ is \textit{abelian} in the statement. With \cite[Proposition 2, III p. 35]{Bou04}, the last part is an easy computation left to the reader.

\begin{prop}
\label{Prop-Takesaki}
Let $M$ be an abelian von Neumann algebra acting on a Hilbert space $H$ equipped with a weak* continuous action $\alpha$ of a locally compact group $G$. Suppose that there exists a strongly continuous function $u \co G \to \mathrm{U}(M)$ such that
\begin{equation}
\label{equation-unitaries}
u(sr)
=u(s)\alpha_s(u(r)),\quad s,r \in G. 
\end{equation}
Then $V \co \L^2(G,H) \to \L^2(G,H)$, $\xi \mapsto (s \mapsto u(s^{-1})(\xi(s)))$ is a unitary and we have a $*$-isomorphism $U \co M \rtimes_{\alpha} G \to  M \rtimes_{\alpha} G$, $x \mapsto VxV^*$ such that
\begin{equation*}
U(\lambda_s \ot \Id_H)
=\pi(u(s)^*)(\lambda_s \ot \Id_H)
\quad \text{and} \quad 
U(\pi(x))
=\pi(x), \quad s \in G, x \in M.
\end{equation*}
Moreover, for any $f \in \mathcal{K}(G,M)$, we have
\begin{equation}
\label{quoi}
U\bigg(\int_G f_s \rtimes \lambda_s \d\mu_G(s) \bigg)
=\int_G u(s)^* f_s  \rtimes \lambda_s \d\mu_G(s).
\end{equation}
\end{prop}

Now, we suppose that the von Neumann algebra $M$ is \textit{finite} and equipped with a normal finite faithful trace $\tau$. By \cite[Lemma 3.3]{Haa78} \cite[Theorem p. 301]{Str81} \cite[Theorem 1.17 p. 249]{Tak03}, there exists a unique normal semifinite faithful weight $\varphi_{\rtimes}$ on the crossed product $M \rtimes_{\alpha} G$ which satisfies for any $f,g \in \mathcal{K}(G,M)$ the  fundamental ``noncommutative Plancherel formula''
\begin{equation}
\label{Plancherel-Non-com}
\varphi_{\rtimes}\bigg(\bigg( \int_G f_s \rtimes \lambda_s \d\mu_G(s)\bigg)^*\bigg( \int_G g_s \rtimes \lambda_s \d\mu_G(s)\bigg)\bigg)
=\int_G \tau(f_s^*g_s) \d\mu_G(s)
\end{equation}
and the relations $\sigma_t^{\varphi_{\rtimes}}(\pi(x))=\pi(x)$ where $x \in M$ and $t \in \R$
and 
\begin{equation*}
\sigma_t^{\varphi_{\rtimes}}(\lambda_s \ot \Id_H)
=\Delta_G^{\i t}(s)(\lambda_s \ot \Id_H)\pi([\D(\tau \circ \alpha_s):\D\tau]_t), \quad s \in G,t \in \R.
\end{equation*}
If $M=\C$, we recover the Plancherel weight $\varphi_G$ on the group von Neumann algebra $\VN(G)$ \cite[p. 67]{Tak03}. If each $\alpha_s \co M \to M$ is trace preserving, we obtain in particular 
\begin{equation*}
\label{modular-group-crossed-prime}
\sigma_t^{\varphi_{\rtimes}}(\lambda_s \ot \Id_H)
=\Delta_G^{\i t}(s)(\lambda_s \ot \Id_H), \quad s \in G,t \in \R.
\end{equation*}
Using \cite[Proposition 2, III p. 35]{Bou04}, we deduce that
\begin{equation}
\label{modular-group-crossed}
\sigma_t^{\varphi_{\rtimes}}\bigg(\int_G f_s \rtimes \lambda_s \d\mu_G(s)\bigg)
=\int_G \Delta_G^{\i t}(s)f_s \rtimes \lambda_s \d\mu_G(s), \quad f \in \mathcal{K}(G,M), t \in \R.
\end{equation}

By \cite[Theorem 4.1]{HJX10}, 
 we have the following result. Note that the proof of \cite[Theorem 4.1]{HJX10} does \textit{not} use the fact that $G$ is abelian. The second part is an obvious observation left to the reader.

\begin{lemma}
\label{Lemma-crossed}
Let $G$ be a locally compact group and $\alpha \co G \to \Aut(M)$ be a weak* continuous action on a von Neumann algebra $M$ equipped with a normal semifinite faithful weight. Let $\E \co M \to M$ be a weight preserving faithful normal conditional expectation such that $
\E\alpha_s=\alpha_s\E$ for any $s \in G$.
\begin{enumerate}
\item There exists a weight preserving faithful normal conditional expectation $
\E \rtimes \Id_{\VN(G)} \co M \rtimes_{\alpha} G \to M \rtimes_{\alpha} G$ 
such that for any $s \in G$ and any $x \in M$
\begin{equation*}
\big(\E \rtimes \Id_{\VN(G)}\big)(\pi(x))=\pi(\E(x)),
\quad 
\big(\E \rtimes \Id_{\VN(G)}\big)(\lambda_s \ot \Id_{H})=\lambda_s \ot \Id_{H}.
\end{equation*}
\item For any function $f \in \mathcal{K}(G,M)$, we have 
\begin{equation}
\label{crossed-2}
\big(\E \rtimes \Id_{\VN(G)}\big)\bigg(\int_G f_s \rtimes \lambda_s \d\mu_G(s) \bigg)
	=\int_G \E(f_s) \rtimes \lambda_s \d\mu_G(s).
\end{equation}
\end{enumerate}
\end{lemma}

\paragraph{Weights}
We will use the following result which is a particular case of \cite[Theorem 6.2 p. 83]{Str81}. Recall that a normal semifinite weight $\psi$ commutes with a normal semifinite faithful weight $\varphi$ if $\psi \circ \sigma_t^\varphi=\psi$ for any $t \in \R$, see \cite[p. 68]{Str81} and that $\frak{n}_\varphi \ov{\mathrm{def}}{=} \{x \in M : \varphi(x^*x) <\infty\}$. 

\begin{lemma}
\label{lemme-egalite-poids}
Let $\varphi$ and $\psi$ two normal semifinite faithful weights on a von Neumann algebra $M$ such that $\psi$ commutes with $\varphi$. Assume that there exists a weak* dense $*$-subalgebra $A$ of $M$ such that $A \subset \frak{n}_\varphi$ which is $\sigma^\varphi$-invariant such that
$$
\psi(x^*x) 
=\varphi(x^*x), \quad x \in A.
$$
Then $\varphi=\psi$.
\end{lemma}

\paragraph{Fourier multipliers} Let $G$ be a locally compact group. We say that a weak* continuous operator $T \co \VN(G)) \to \VN(G)$ is a Fourier multiplier if there exists a continuous function $\phi \co G \to \C$ such that for any $s \in G$ we have $T(\lambda_s)=\phi(s)\lambda_s$. In this case, $\phi$ is bounded and for any function $f \in \mathrm{C}_c(G)$ the element $\int_{G} \phi(s) f(s) \lambda_s \d \mu_G(s)$ belongs to the von Neumann algebra $\VN(G)$ and 
\begin{equation}
\label{equ-def-Fourier-mult}
T\bigg(\int_{G} f(s) \lambda_s \d \mu_G(s)\bigg) 
= \int_{G} \phi(s) f(s) \lambda_s \d \mu_G(s).
\end{equation}
In this case, we let $M_\phi \ov{\mathrm{def}}{=} T$ and we say that $\phi$ is the symbol of $T$. We refer to the books \cite{KaL18} and \cite{ArK20} and references therein for more information.

\paragraph{Semigroups of Fourier multipliers} Consider a locally compact group $G$ with identity element $e$. Let $(T_t)_{t \geq 0}$ be a weak* continuous semigroup of selfadjoint unital completely positive Fourier multipliers. There exists a (unique) continuous real-valued conditionally negative definite function $\psi \co G \to \R$ satisfying $\psi(e) = 0$ such that 
$$
T_t(\lambda_s) 
= \e^{-t \psi(s)} \lambda_s, \quad t \geq 0, s \in G.
$$
In this case, there exists a real Hilbert space $H$ together with a mapping $b_\psi \co G \to H$ and a homomorphism $\pi \co G \to \mathrm{O}(H)$ such that the $1$-cocycle law holds 
\begin{equation}
\label{Cocycle-law}
\pi_s(b_\psi(r))
=b_\psi(sr)-b_\psi(s),
\quad \text{i.e.} \quad 
b_\psi(sr)
=b_\psi(s)+\pi_s(b_\psi(r))
\end{equation}
for any $s,r \in G$ and such that 
\begin{equation}
\label{liens-psi-bpsi}
\psi(s)
=\|b_\psi(s)\|_H^2, \quad s \in G.
\end{equation}
We refer to the book \cite{BHV08} for more information on affine isometric actions of groups and 1-cocycles.

\section{Markov dilations of semigroups of Fourier multipliers}
\label{sec-Markov}



Our main result is the following theorem which gives a standard Markov dilation. Here, we equip the von Neumann algebra $\VN(G)$ with the Plancherel weight.

\begin{thm}
\label{Th-Markov-dilation-Fourier}
Let $G$ be a locally compact group. Consider a weak* continuous semigroup $(T_t)_{t \geq 0}$ of selfadjoint unital completely positive Fourier multipliers on $\VN(G)$ defined by \eqref{liens-psi-bpsi}. Then there exists a von Neumann algebra $M$ equipped with a normal semifinite faithful weight $\varphi_M$, an increasing filtration $(M_t)_{t \geq 0}$ of the algebra $M$ with associated weight preserving normal faithful conditional expectations $\E_t \co M \to M_t$ and weight preserving unital normal injective $*$-homomorphisms $\pi_t \co \VN(G) \to M_t$ such that
\begin{equation}
\label{Equa-standard-Markov}
\E_{u}\pi_t
=\pi_{u}T_{t-u}, \quad 0 \leq u \leq t.
\end{equation} 
Moreover, we have the following properties.
\begin{enumerate}
	\item If $G$ is discrete then the weight $\varphi_M$ is a normal finite faithful trace.
	\item If $G$ is unimodular then the weight $\varphi_M$ is a normal semifinite faithful trace.
	\item If $G$ is amenable then the von Neumann algebra $M$ is injective.
\end{enumerate}
\end{thm}

\begin{proof}
Here, we suppose that $H$, $\pi$ and $b_\psi$ are defined as in \eqref{Cocycle-law}. Let $\W \co \L^2_\R(\R^+,H) \to \L^0(\Omega)$ be an $H$-cylindrical Brownian motion on a probability space $(\Omega,\mu)$, see Section \ref{sec:preliminaries}. For any $s \in G$, we will use the second quantization $\alpha_s \ov{\mathrm{def}}{=} \Gamma^\infty(\Id_{\L^2_\R(\R^+)} \ot \pi_s) \co \L^\infty(\Omega) \to \L^\infty(\Omega)$ which is integral preserving. In particular, if $r,s \in G$ and if $t \geq 0$, we have
\begin{equation}
\label{Action-on-bpsi}
\alpha_s\big(\e^{-\sqrt{2}\i \W_t(b_\psi(r))}\big)
=\Gamma^\infty(\Id_{\L^2_\R(\R^+)} \ot \pi_s)\big(\e^{-\sqrt{2}\i \W_t(b_\psi(r))}\big)
\ov{\eqref{SQ1}\eqref{Def-Wt-h}}{=} \e^{-\sqrt{2}\i \W_t(\pi_s(b_\psi(r)))}.
\end{equation}
Since the orthogonal representation $\pi$ is strongly continuous, we obtain by Lemma \ref{Lemma-semigroup-continuous} a continuous action $\alpha \co G \to \Aut(\L^\infty(\Omega))$. So we can consider the crossed product $M \ov{\mathrm{def}}{=} \L^\infty(\Omega) \rtimes_{\alpha} G$ equipped with its canonical normal semifinite faithful weight $\varphi_M \ov{\mathrm{def}}{=} \varphi_{\rtimes}$. We denote by $J \co \VN(G) \to \L^\infty(\Omega) \rtimes_{\alpha} G$ the canonical unital normal injective $*$-homomorphism. Using \cite[Proposition 2, III p. 35]{Bou04}, for any $f \in \mathrm{C}_c(G)$, we see that
\begin{equation}
\label{Def-de-J-prem}
J\bigg(\int_G f(s) \lambda_s \d\mu_G(s)\bigg)
=\int_G f(s)1 \rtimes \lambda_s \d\mu_G(s).
\end{equation}
The same proof as the one of \cite[Lemma 3.2]{Arh21}, shows that the map $J$ is weight preserving. For any $t \in \R$, we consider the function $u_t \co G \to \mathrm{U}(\L^\infty(\Omega))$, $s \mapsto \e^{-\sqrt{2}\i \W_t(b_\psi(s))}$. The map $b_\psi \co G \to H$ is continuous. By the first point of Lemma \ref{Lemma-semigroup-continuous}, the map $\L^2_\R(\R^+,H) \to \L^\infty(\Omega)$, $g \mapsto \e^{\i\W(g)}$ is continuous if $\L^\infty(\Omega)$ is equipped with the weak* topology, hence with the weak operator topology when we consider that the von Neumann algebra $\L^\infty(\Omega)$ acts on $\L^2(\Omega)$. Recall that by \cite[Exercice 5.7.5]{KaR97} or \cite[p. 41]{Str81} the weak operator topology and the strong operator topology coincide on the unitary group $\mathrm{U}(\L^\infty(\Omega))$. So by composition, the function $u_t$ is continuous if $\mathrm{U}(\L^\infty(\Omega))$ is equipped with the strong operator topology. For any $t \geq 0$ and any $r,s \in G$, note that  
\begin{align*}
\MoveEqLeft
u_t(sr)
=\e^{-\sqrt{2}\i  \W_t(b_\psi(sr))}
\ov{\eqref{Cocycle-law}}{=}
\e^{-\sqrt{2}\i  \W_t(b_\psi(s))}\e^{-\sqrt{2}\i  \W_t(\pi_s(b_\psi(r)))}\\
&\ov{\eqref{Action-on-bpsi}}{=}u_t(s)\alpha_s\big(\e^{-\sqrt{2}\i \W_t(b_\psi(r))}\big)
=u_t(s)\alpha_s(u_t(r)).            
\end{align*}
Hence \eqref{equation-unitaries} is satisfied. By Proposition \ref{Prop-Takesaki}, for any $t \geq 0$, we have a unitary $V_t \co \L^2(G,\L^2(\Omega)) \to \L^2(G,\L^2(\Omega))$, $\xi \mapsto (s \mapsto u_t(s^{-1})(\xi(s)))$ and a $*$-isomorphism 
\begin{equation*}
\begin{array}{cccc}
U_t \co   &  \L^\infty(\Omega) \rtimes_\alpha G   &  \longrightarrow   &  \L^\infty(\Omega) \rtimes_\alpha G  \\
        &  x   &  \longmapsto       &  V_t xV_t^* \\
\end{array}
\end{equation*}
such that for any function $f \in \mathcal{K}(G,\L^\infty(\Omega))$
\begin{equation}
\label{Def-Ut}
U_t\bigg(\int_G f_s \rtimes \lambda_s \d\mu_G(s)\bigg)
= \int_G \e^{\sqrt{2}\i \W_t(b_\psi(s))}f_s  \rtimes \lambda_s \d\mu_G(s), \quad t \geq 0. 
\end{equation}

\begin{lemma}
For any $t \geq 0$, the map $U_t$ is weight preserving.
\end{lemma}

\begin{proof}
We will use Lemma \ref{lemme-egalite-poids} with the weights $\varphi_\rtimes$ and $\varphi_\rtimes \circ U_t$ on $\L^\infty(\Omega) \rtimes_\alpha G$. Note that the space of elements $\int_G f_s \rtimes \lambda_s \d\mu_G(s)$ for $f \in \mathcal{K}(G,\L^\infty(\Omega))$ is a $*$-subalgebra which is $\sigma^{\varphi_\rtimes}$-invariant by \eqref{modular-group-crossed}, weak* dense in $\L^\infty(\Omega) \rtimes_{\alpha} G$ and included in $\mathfrak{n}_{\varphi_\rtimes}$. The formulas \eqref{modular-group-crossed} and \eqref{Def-Ut} show that each $U_t$ and $\sigma_t^{\varphi_\rtimes}$ commute. So, we have
$$
\varphi_\rtimes \circ U_t \circ \sigma_t^{\varphi_\rtimes}
=\varphi_\rtimes \circ \sigma_t^{\varphi_\rtimes} \circ U_t
=\varphi_\rtimes \circ U_t.
$$
So the weights $\varphi_\rtimes \circ U_t$ and $\varphi_\rtimes$ commutes by \cite[pp. 67-68]{Str81}. It is easy to check that the weight $\varphi_\rtimes \circ U_t$ is normal and semifinite. If $f \in \mathcal{K}(G,\L^\infty(\Omega))$, we have
\begin{align*}
\MoveEqLeft
\varphi_\rtimes  \circ U_t\bigg(\bigg(\int_G f_s \rtimes \lambda_s \d\mu_G(s)\bigg)^*\bigg(\int_G f_s \rtimes \lambda_s \d\mu_G(s)\bigg)\bigg) \\          
		&=\varphi_\rtimes\Bigg(\bigg(U_t\bigg(\int_G f_s \rtimes \lambda_s \d\mu_G(s)\bigg)\bigg)^*U_t\bigg(\int_G f_s \rtimes \lambda_s \d\mu_G(s)\bigg)\Bigg) \\
		&\ov{\eqref{Def-Ut}}{=}\varphi_\rtimes\Bigg(\bigg(\int_G \e^{\sqrt{2}\i \W_t(b_\psi(s))}f_s  \rtimes \lambda_s \d\mu_G(s)\bigg)^*\bigg(\int_G \e^{\sqrt{2}\i \W_t(b_\psi(s))}f_s  \rtimes \lambda_s \d\mu_G(s)\bigg)\Bigg)\\
		&\ov{\eqref{Plancherel-Non-com}}{=} \int_G \int_\Omega\e^{-\sqrt{2}\i \W_t(b_\psi(s))}f_s^*\e^{\sqrt{2}\i \W_t(b_\psi(s))}f_s \d \mu \d\mu_G(s)
		=\int_G \int_\Omega f_s^*f_s\d \mu \d\mu_G(s)\\
		&\ov{\eqref{Plancherel-Non-com}}{=} \varphi_\rtimes\bigg(\bigg(\int_G f_s \rtimes \lambda_s \d\mu_G(s)\bigg)^*\bigg(\int_G f_s \rtimes \lambda_s \d\mu_G(s)\bigg)\bigg).
\end{align*} 
We conclude with Lemma \ref{lemme-egalite-poids} that $\varphi_\rtimes \circ U_t=\varphi_\rtimes$ for any $t \geq 0$.
\end{proof}

For any $t \geq 0$, we define the unital normal injective $*$-homomorphism
\begin{equation}
\label{def-Pi-s-Schur-2}
\pi_t\ov{\mathrm{def}}{=} U_tJ \co     \VN(G)     \to  \L^\infty(\Omega) \rtimes_\alpha G .	
\end{equation}
Each $\pi_t$ is weight preserving by composition. For any $t \geq 0$, we also define the canonical normal conditional expectations $\E_{\scr{F}_t} \co \L^\infty(\Omega) \to \L^\infty(\Omega)$ on $\L^\infty(\Omega,\scr{F}_t)$ where the $\sigma$-algebra $\scr{F}_t$ is defined in \eqref{Filtration-H-cylindrical}. Recall that $(\W_t(h))_{t \geq 0}$ is a Brownian motion for any fixed $h \in H$. Hence for any $0 \leq u \leq t$ and any $s \in G$ the random variable 
\begin{align}
\MoveEqLeft
\label{Eq-inter-1}
\W\big(1_{]u,t]} \ot b_\psi(s)\big)
=\W\big(1_{[0,t]} \ot b_\psi(s)\big)-\W\big(1_{[0,u]} \ot b_\psi(s)\big) 
\ov{\eqref{Def-Wt-h}}{=} \W_t(b_\psi(s))-\W_u(b_\psi(s)) 
\end{align}
is independent by \eqref{difference-3} from the $\sigma$-algebra $\scr{F}_u \ov{\eqref{Filtration-H-cylindrical}}{=} \sigma\big(\W_r(h) : r \in [0,u], h \in H \big)$. Consequently, the random variable $\e^{\sqrt{2}\i\W(1_{]u,t]}\ot b_\psi(s))}$ is also independent from the $\sigma$-algebra $\scr{F}_u$.

\begin{lemma}
The $\sigma$-algebra $\scr{F}_u$ is equal to the $\sigma$-algebra $\scr{G}$ generated by the random variables $\W(g)$ where $g \in \L^2_\R([0,u],H)$.
\end{lemma}

\begin{proof}
It suffices to show that the space $\L^\infty(\Omega,\scr{F}_u)$ is weak* dense in $\L^\infty(\Omega,\scr{G})$. Note that by \cite[Remark 1.2.20 p.~24]{HvNVW16} the subspace $E$ of elements $\sum_{k=1}^{m} 1_{[c_k,d_k]} \ot h_k$, where $0 \leq c_1 < d_1 < c_2 < d_2 < \cdots < c_m < d_m \leq u$ and $h_1,\ldots,h_m \in H$, is dense in the real Hilbert space $\L^2_\R([0,u],H)$. Hence by the discussion after \eqref{Concrete-W} the span of elements $\e^{\i\W(f)}$ where $f \in E$ is weak* dense in the space $\Gamma_1(\L^2_\R([0,u],H))$, which identifies to $\L^\infty(\Omega,\scr{G})$. We can conclude since $\L^\infty(\Omega,\scr{F}_u)$ contains these elements.
\end{proof}

Consider the projection $P_u \co \L^2_\R(\R^+,H) \to \L^2_\R(\R^+,H)$ on the closed subspace $\L^2_\R([0,u],H)$. By the previous lemma and by an observation following \eqref{SQ1}, we see that $\E_{\scr{F}_u}=\Gamma^\infty(P_u \ot \Id_H)$ for any $u \geq 0$. Consequently, for any $s \in G$ and any $u \geq 0$, we obtain that
\begin{align*}
\MoveEqLeft
\alpha_s\E_{\scr{F}_u}
=\Gamma^\infty(\Id_{\L^2_\R(\R^+)} \ot \pi_s)\Gamma^\infty(P_u \ot \Id_H)
=\Gamma^\infty(P_u \ot \pi_s) \\
&=\Gamma^\infty(P_u \ot \Id_H) \Gamma^\infty(\Id_{\L^2_\R(\R^+)} \ot \pi_s)
=\E_{\scr{F}_u} \alpha_s.         
\end{align*}
So by Lemma \ref{Lemma-crossed}, we can consider the map $\E_t \ov{\mathrm{def}}{=} \E_{\scr{F}_t} \rtimes \Id_{\VN(G)} \co \L^\infty(\Omega) \rtimes_\alpha G \to \L^\infty(\Omega,\scr{F}_t) \rtimes_\alpha G$. We introduce the von Neumann algebra $N_t \ov{\mathrm{def}}{=} \L^\infty(\Omega,\scr{F}_t) \rtimes_\alpha G$. Moreover, we have
\begin{align}
\MoveEqLeft
\label{particular}
\E_{\scr{F}_u}\big(\e^{\sqrt{2}\i  \W_t(b_\psi(s))}\big)
\ov{\eqref{Bimodule}}{=} \e^{\sqrt{2}\i  \W_u(b_\psi(s))} \E_{\scr{F}_u}\big(\e^{\sqrt{2}\i \W(1_{[u,t]} \ot b_\psi(s))}\big)\\
&\ov{\eqref{Hy-2.6.35}}{=} \e^{\sqrt{2}\i  \W_u(b_\psi(s))} \E\big(\e^{\sqrt{2}\i \W(1_{[u,t]} \ot b_\psi(s))}\big) 
\ov{\eqref{Esperance-exponentielle-complexe-2}}{=}\e^{-(t-u)\norm{b_\psi(s)}_{H}^2} \e^{\sqrt{2}\i  \W_u(b_\psi(s))}. \nonumber
\end{align}
For any function $f \in \mathrm{C}_c(G)$ and any $t \geq 0$, we have
\begin{align}
\MoveEqLeft
\label{First-calcul}
\pi_t\bigg(\int_G f(s) \lambda_s \d\mu_G(s)\bigg)          
\ov{\eqref{def-Pi-s-Schur-2}}{=} U_tJ\bigg(\int_G f(s) \lambda_s \d\mu_G(s)\bigg) \\
&=U_t\bigg(\int_G f(s)1 \rtimes \lambda_s \d\mu_G(s)\bigg)
\ov{\eqref{Def-Ut}}{=} \int_G f(s)\e^{\sqrt{2}\i  \W_t(b_\psi(s))}  \rtimes \lambda_s \d\mu_G(s). \nonumber
\end{align} 
Similarly, for any $0 \leq u \leq t$, we have
\begin{align}
\MoveEqLeft
\label{Calcul-2}
\pi_uT_{t-u}\bigg(\int_G f(s) \lambda_s \d\mu_G(s)\bigg)           
\ov{\eqref{equ-def-Fourier-mult}}{=} \pi_u\bigg(\int_G \e^{-(t-u)\norm{b_\psi(s)}_{H}^2} f(s) \lambda_s \d\mu_G(s)\bigg) \\
&\ov{\eqref{First-calcul}}{=} \int_G f(s)\e^{-(t-u)\norm{b_\psi(s)}_{H}^2}\e^{\sqrt{2}\i  \W_u(b_\psi(s))}  \rtimes \lambda_s \d\mu_G(s). \nonumber
\end{align} 
We finally obtain for any $0 \leq u \leq t$ and any function $f \in \mathrm{C}_c(G)$
\begin{align*}
\MoveEqLeft
\E_u\pi_t\bigg(\int_G f(s) \lambda_s \d\mu_G(s)\bigg)              
\ov{\eqref{First-calcul}}{=} \E_u\bigg( \int_G f(s)\e^{\sqrt{2}\i\W_t(b_\psi(s))}  \rtimes \lambda_s \d\mu_G(s)\bigg) \\
&\ov{\eqref{crossed-2}}{=} \int_G f(s)\E_{\scr{F}_u}\big(\e^{\sqrt{2}\i\W_t(b_\psi(s))} \big) \rtimes \lambda_s \d\mu_G(s) \\
&\ov{\eqref{particular}}{=} \int_G f(s) \e^{-(t-u)\norm{b_\psi(s)}_{H}^2}\e^{\sqrt{2}\i\W_u(b_\psi(s))}  \rtimes \lambda_s \d\mu_G(s) \\
&\ov{\eqref{Calcul-2}}{=} \pi_u T_{t-u}\bigg(\int_G f(s) \lambda_s \d\mu_G(s)\bigg) .
\end{align*} 
By weak* density, the proof is complete.

Now, we prove the last assertions. Note each $\alpha_s \co \L^\infty(\Omega) \to \L^\infty(\Omega)$ preserves the integral. The first is well-known, e. g. \cite[Corollary 7.11.8]{Ped79}. The second is folklore. The third is \cite[Proposition p. 301]{Ana79}.
\end{proof}


Similarly, we can prove the following reversed Markov dilation.

\begin{thm}
\label{Th-reversed-Markov-dilation-Schur}
Let $G$ be a locally compact group. Consider a weak* continuous semigroup $(T_t)_{t \geq 0}$ of selfadjoint unital completely positive Fourier multipliers on $\VN(G)$ defined by \eqref{liens-psi-bpsi}. There exists a von Neumann algebra $M$ equipped with a normal semifinite faithful weight, a decreasing filtration $(M_t)_{t \geq 0}$ of $M$ with associated weight preserving normal faithful conditional expectations $\E_t \co M \to M_t$ and weight preserving unital normal injective $*$-homomorphisms $\pi_t \co \VN(G) \to M_t$ such that
\begin{equation}
\label{-equa-reversed-Markov}
\E_u\pi_t
=\pi_uT_{u-t}, \quad 0 \leq t \leq u.
\end{equation}
\end{thm}

\textbf{Acknowledgements}.
The author acknowledges support by the grant ANR-18-CE40-0021 (project HASCON) of the French National Research Agency ANR. The author is grateful to the referee for some corrections and a simplification of the proof.


\vspace{0.2cm}
\footnotesize{
\noindent C\'edric Arhancet\\ 
\noindent 6 rue Didier Daurat, 81000 Albi, France\\
URL: \href{http://sites.google.com/site/cedricarhancet}{https://sites.google.com/site/cedricarhancet}\\
cedric.arhancet@protonmail.com\\

\normalsize

\end{document}

\newpage 

\begin{proof}
Note that by \cite[Remark 1.2.20]{HvNVW16} the space of elements $\sum_{k=1}^{m} 1_{[c_k,d_k]} \ot h_k$, where $c_1 < d_1 < c_2 < d_2 < \cdots < c_m < d_m$ and $h_1,\ldots,h_m \in H$, is dense in the real Hilbert space $\L^2_\R(\R^+,H)$. We infer by weak* density that it suffices to show that 
\begin{equation}
\label{ultime}
\alpha_s\E_{\scr{F}_u}\big(\e^{\sqrt{2}\i  \W(\sum_{k=1}^{m}1_{[c_k,d_k]} \ot h_k)}\big)
=\E_{\scr{F}_u}\alpha_s\big(\e^{\sqrt{2}\i  \W(\sum_{k=1}^{m} 1_{[c_k,d_k]} \ot h_k)}\big).
\end{equation}
We reason by cases.

\textit{First case} Suppose that $d_n < u < c_{n+1}$ for some $n$. On the one hand, using the $\scr{F}_u$-measurability of the random variable $\e^{\sum_{k=1}^{n}\sqrt{2}\i  \W(1_{[c_k,d_k]} \ot h_k)}$ in the second equality, we obtain
\begin{align}
\MoveEqLeft
\alpha_s\E_{\scr{F}_u}\big(\e^{\sqrt{2}\i  \W(\sum_{k=1}^{m}1_{[c_k,d_k]} \ot h_k)}\big)    
=\alpha_s\E_{\scr{F}_u}\big(\e^{\sum_{k=1}^{m}\sqrt{2}\i  \W(1_{[c_k,d_k]} \ot h_k)}\big) \nonumber\\   
&\ov{\eqref{Bimodule}}{=} \alpha_s\Big(\e^{\sum_{k=1}^{n}\sqrt{2}\i  \W(1_{[c_k,d_k]} \ot h_k)}\E_{\scr{F}_u}\big(\e^{\sum_{k=n+1}^{m}\sqrt{2}\i  \W(1_{[c_k,d_k]} \ot h_k)}\big)\Big) \nonumber \\
&=\e^{\sum_{k=1}^{n}\sqrt{2}\i  \W(1_{[c_k,d_k]} \ot \pi_s(h_k))}\alpha_s\Big(\E_{\scr{F}_u}\big(\prod_{k=n+1}^{m}\e^{\sqrt{2}\i  \W(1_{[c_k,d_k]} \ot h_k)}\big)\Big)\nonumber\\
&\ov{\eqref{Hy-2.6.35}}{=} \prod_{k=1}^{n}\e^{\sqrt{2}\i  \W(1_{[c_k,d_k]} \ot \pi_s(h_k))}\alpha_s\Big(\E\big(\prod_{k=n+1}^{m}\e^{\sqrt{2}\i  \W(1_{[c_k,d_k]} \ot h_k)}\big)\Big)\nonumber\\
&=\Big(\prod_{k=1}^{n}\e^{\sqrt{2}\i  \W(1_{[c_k,d_k]} \ot \pi_s(h_k))}\Big) \Big(\prod_{k=n+1}^{m}\E\big(\e^{\sqrt{2}\i  \W(1_{[c_k,d_k]} \ot h_k)}\big)\Big) \nonumber \\
&\ov{\eqref{Esperance-exponentielle-complexe-2}}{=} \Big(\prod_{k=1}^{n}\e^{\sqrt{2}\i  \W(1_{[c_k,d_k]} \ot \pi_s(h_k))}\Big) \Big(\prod_{k=n+1}^{m} \e^{-\frac{d_k-c_k}{2} \norm{h_k}_H^2}\Big). \nonumber
\end{align} 
On the other hand, we have
\begin{align}
\MoveEqLeft
\E_{\scr{F}_u}\alpha_s\big(\e^{\sqrt{2}\i  \W(\sum_{k=1}^{m}1_{[c_k,d_k]} \ot h_k)}\big)    
=\E_{\scr{F}_u}\big(\e^{\sum_{k=1}^{m}\sqrt{2}\i  \W(1_{[c_k,d_k]} \ot \pi_s(h_k))}\big) \nonumber \\  
&\ov{\eqref{Bimodule}}{=} \e^{\sum_{k=1}^{n}\sqrt{2}\i  \W(1_{[c_k,d_k]} \ot \pi_s(h_k))}\E_{\scr{F}_u}\bigg(\prod_{k=n+1}^{m}\e^{\sqrt{2}\i  \W(1_{[c_k,d_k]} \ot \pi_s(h_k))}\bigg) \nonumber\\
&\ov{\eqref{Hy-2.6.35}}{=} \prod_{k=1}^{n}\e^{\sqrt{2}\i  \W(1_{[c_k,d_k]} \ot \pi_s(h_k))}\E\bigg(\prod_{k=n+1}^{m}\e^{\sqrt{2}\i  \W(1_{[c_k,d_k]} \ot \pi_s(h_k))}\bigg) \nonumber\\
&=\prod_{k=1}^{n}\e^{\sqrt{2}\i  \W(1_{[c_k,d_k]} \ot \pi_s(h_k))}\bigg(\prod_{k=n+1}^{m}\E\big(\e^{\sqrt{2}\i  \W(1_{[c_k,d_k]} \ot \pi_s(h_k))}\big) \bigg) \nonumber \\ 
&\ov{\eqref{Esperance-exponentielle-complexe-2}}{=} \Big(\prod_{k=1}^{n}\e^{\sqrt{2}\i  \W(1_{[c_k,d_k]} \ot \pi_s(h_k))}\Big) \Big(\prod_{k=n+1}^{m} \e^{-\frac{d_k-c_k}{2} \norm{\pi_s(h_k)}_H^2}\Big) \nonumber \\
&=\Big(\prod_{k=1}^{n}\e^{\sqrt{2}\i  \W(1_{[c_k,d_k]} \ot \pi_s(h_k))}\Big) \Big(\prod_{k=n+1}^{m} \e^{-\frac{d_k-c_k}{2} \norm{h_k}_H^2}\Big). \nonumber
\end{align}

\textit{Second case} Suppose that $c_n  < u < d_n$ for some $n$. On the one hand, we have
\begin{align*}
\MoveEqLeft
\E_{\scr{F}_u}\alpha_s\big(\e^{\sqrt{2}\i  \W(\sum_{k=1}^{m}1_{[c_k,d_k]} \ot h_k)}\big) \\   
&=\E_{\scr{F}_u}\big(\e^{\sqrt{2}\i  \W(\sum_{k=1}^{m}1_{[c_k,d_k]} \ot \pi_s(h_k))}\big) \\
&=\E_{\scr{F}_u}\bigg(\prod_{k=1}^{n-1}\e^{\sqrt{2}\i  \W(1_{[c_k,d_k]} \ot \pi_s(h_k))}\e^{\sqrt{2}\i  \W(1_{[c_n,u]} \ot \pi_s(h_n))}\e^{\sqrt{2}\i  \W(1_{[u,d_n]} \ot  \pi_s(h_n))}\\
&\qquad\qquad\qquad\qquad\qquad\qquad\qquad\qquad\qquad\qquad\qquad \times \prod_{k=n+1}^{m}\e^{\sqrt{2}\i  \W(1_{[c_k,d_k]} \ot \pi_s(h_k))}\bigg) \\
&\ov{\eqref{Bimodule}}{=} \prod_{k=1}^{n-1}\e^{\sqrt{2}\i  \W(1_{[c_k,d_k]} \ot \pi_s(h_k))}\e^{\sqrt{2}\i  \W(1_{[c_n,u]} \ot \pi_s(h_n))}\E_{\scr{F}_u}\bigg(\e^{\sqrt{2}\i  \W(1_{[u,d_n]} \ot  \pi_s(h_n))} \bigg.\\
&\qquad \qquad \qquad \qquad \qquad \qquad \qquad\qquad \qquad \qquad \qquad \times \bigg.\prod_{k=n+1}^{m}\e^{\sqrt{2}\i  \W(1_{[c_k,d_k]} \ot \pi_s(h_k))}\bigg) \\
&\ov{\eqref{Hy-2.6.35}}{=}\prod_{k=1}^{n-1}\e^{\sqrt{2}\i  \W(1_{[c_k,d_k]} \ot \pi_s(h_k))}\e^{\sqrt{2}\i  \W(1_{[c_n,u]} \ot \pi_s(h_n))}\E\bigg(\e^{\sqrt{2}\i  \W(1_{[u,d_n]} \ot  \pi_s(h_n))}\bigg. \\
&\qquad\qquad\qquad\qquad\qquad\qquad\qquad\qquad\qquad\qquad\qquad\qquad\bigg. \times \prod_{k=n+1}^{m}\e^{\sqrt{2}\i  \W(1_{[c_k,d_k]} \ot \pi_s(h_k))}\bigg) \\
&=\prod_{k=1}^{n-1}\e^{\sqrt{2}\i  \W(1_{[c_k,d_k]} \ot \pi_s(h_k))}\e^{\sqrt{2}\i \W(1_{[c_n,u]} \ot \pi_s(h_n))}\E\big(\e^{\sqrt{2}\i  \W(1_{[u,d_n]} \ot \pi_s(h_n))}\big)\\
&\qquad\qquad\qquad\qquad\qquad\qquad\qquad\qquad\qquad\qquad\qquad \times \bigg(\prod_{k=n+1}^{m}\E\bigg(\e^{\sqrt{2}\i  \W(1_{[c_k,d_k]} \ot \pi_s(h_k))}\bigg)\\
&\ov{\eqref{Esperance-exponentielle-complexe-2}}{=} \prod_{k=1}^{n-1}\e^{\sqrt{2}\i  \W(1_{[c_k,d_k]} \ot \pi_s(h_k))}\e^{\sqrt{2}\i  \W(1_{[c_n,u]} \ot \pi_s(h_n))} \e^{-\frac{d_n-u}{2} \norm{\pi_s(h_k)}_H^2}\\
&\qquad\qquad\qquad\qquad\qquad\qquad\qquad\qquad\qquad\qquad\qquad\qquad \times \bigg(\prod_{k=n+1}^{m}\e^{-\frac{d_k-c_k}{2} \norm{\pi_s(h_k)}_H^2}\bigg) \\
&=\prod_{k=1}^{n-1}\e^{\sqrt{2}\i  \W(1_{[c_k,d_k]} \ot \pi_s(h_k))}\e^{\sqrt{2}\i  \W(1_{[c_n,u]} \ot \pi_s(h_n))} \e^{-\frac{d_n-u}{2} \norm{h_n}_H^2}\bigg(\prod_{k=n+1}^{m}\e^{-\frac{d_k-c_k}{2} \norm{h_k}_H^2}\bigg).   
\end{align*} 
On the other hand, we have\begin{align*}
\MoveEqLeft
\alpha_s\E_{\scr{F}_u}\big(\e^{\sqrt{2}\i  \W(\sum_{k=1}^{m}1_{[c_k,d_k]} \ot h_k)}\big)    
=\alpha_s\E_{\scr{F}_u}\bigg(\prod_{k=1}^{m}\e^{\sqrt{2}\i  \W(1_{[c_k,d_k]} \ot h_k)}\bigg) \\   
&=\alpha_s\E_{\scr{F}_u}\bigg(\prod_{k=1}^{n-1}\e^{\sqrt{2}\i  \W(1_{[c_k,d_k]} \ot h_k)}\e^{\sqrt{2}\i  \W(1_{[c_n,u]} \ot h_n)}\e^{\sqrt{2}\i  \W(1_{[u,d_n]} \ot h_n)}\prod_{k=n+1}^{m}\e^{\sqrt{2}\i  \W(1_{[c_k,d_k]} \ot h_k)}\bigg) \\
&\ov{\eqref{Bimodule}}{=} \alpha_s\bigg(\prod_{k=1}^{n-1}\e^{\sqrt{2}\i  \W(1_{[c_k,d_k]} \ot h_k)}\e^{\sqrt{2}\i  \W(1_{[c_n,u]} \ot h_n)}\E_{\scr{F}_u}\bigg(\e^{\sqrt{2}\i  \W(1_{[u,d_n]} \ot h_n)}\bigg.\bigg.\\
& \qquad \qquad \qquad\qquad\qquad\qquad\qquad\qquad\qquad\qquad\qquad\qquad\bigg.\bigg.\prod_{k=n+1}^{m}\e^{\sqrt{2}\i  \W(1_{[c_k,d_k]} \ot h_k)}\bigg)\bigg)\\
&\ov{\eqref{Hy-2.6.35}}{=} \prod_{k=1}^{n-1}\e^{\sqrt{2}\i  \W(1_{[c_k,d_k]} \ot \pi_s(h_k))}\e^{\sqrt{2}\i  \W(1_{[c_n,u]} \ot \pi_s(h_n))}\E\bigg(\e^{\sqrt{2}\i  \W(1_{[u,d_n]} \ot h_n)}\bigg.\\
&\qquad \qquad\qquad\qquad\qquad\qquad\qquad\qquad\qquad\qquad\qquad\qquad\bigg.\prod_{k=n+1}^{m}\e^{\sqrt{2}\i  \W(1_{[c_k,d_k]} \ot h_k)}\bigg) \\
&\ov{\eqref{Esperance-exponentielle-complexe-2}}{=} \prod_{k=1}^{n-1}\e^{\sqrt{2}\i  \W(1_{[c_k,d_k]} \ot \pi_s(h_k))}\e^{\sqrt{2}\i  \W(1_{[c_n,u]} \ot \pi_s(h_n))}\e^{-\frac{d_n-u}{2} \norm{h_n}_H^2} \prod_{k=n+1}^{m} \e^{-\frac{d_k-c_k}{2} \norm{h_k}_H^2}.
\end{align*} 
The proof is complete.
\end{proof}

\end{document}

\section{Discussion}

\subsection{La formule du produit scalaire}
$$
\tr(K_f K_g)
=\int_{X \times X} f(x,y) g(y,x) \d x \d y.
$$
The selfadjointness property is equivalent to
$$
\tr\big(M_\varphi(K_f)K_g^*\big)
=\tr\big(K_f(M_\varphi(K_g))^*\big), \quad \text{i.e.} \quad \tr\big(K_{\varphi f}K_g^*\big)
=\tr\big(K_f(K_{\varphi g})^*\big)
$$
for any $f,g \in \L^2(X \times X)$. Since $\L^2(X \times X) \to S^2_X$, $f \mapsto K_f$ is an isometry from the Hilbert space $\L^2(X \times X)$ onto the Hilbert space $S^2_X$ of Hilbert-Schmidt operators on $\L^2(X)$, that means that
$$
\iint_{X \times X} \varphi(x,y)f(x,y)\ovl{g(x,y)} \d x\d y
=\iint_{X \times X} f(x,y) \ovl{\varphi(x,y)g(x,y)} \d x\d y.
$$

\subsection{Fukumizu }
Kenji Fukumizu (demander ref) en fait je crois que c'est fit dans Cartier

\newpage

\subsection{Convolution groupoids}

Now, we prove the converse. 
Suppose that $M_\varphi$ is trace preserving. Let $K_g$ be an element of $S^2_X$ with a continuous kernel $g \in \L^2(X \times X)$. Then if we define $h(x,y) \ov{\mathrm{def}}{=} \int_{X} \ovl{g(z,x)}g(z,y) \d z$ then by \eqref{composition-Hilbert-Schmidt} the composition $K_h=K_g^*K_g$ is a trace-class positive operator. Since $g$ is continuous, its kernel is continuous. Hence
\begin{align*}
\MoveEqLeft
\int_X \varphi(x,x) h(x,x) \d x
\ov{\eqref{trace-of-trace-class}}{=} \tr K_{\varphi h}
=\tr M_\varphi(K_h)            
=\tr K_h 
\ov{\eqref{trace-of-trace-class}}{=} \int_X h(x,x) \d x.
\end{align*}
That means that
$$
\int_{X} \varphi(x,x) \int_{X}\ovl{g(z,x)}g(z,x) \d z \d x
=\int_{X}\int_{X}\ovl{g(z,x)}g(z,x) \d z \d x.
$$
We obtain
$$
\int_{X} \varphi(x,x) (\check{\ovl{g}}*g)(x,x) \d x
=\int_{X} (\check{\ovl{g}}*g)(x,x) \d x.
$$
for the convolution on the groupoid.

\begin{prop}
\label{Prop-unital}
Let $X$ be a second countable compact space equipped with a Radon measure with support $X$. Let $M_\varphi \co \B(\L^2(X)) \to \B(\L^2(X))$ be a measurable Schur multiplier with a continuous symbol $\varphi \co X \times X \to \mathbb{C}$. If $M_\varphi$ is unital then for any $x \in X$ we have $\varphi(x,x)=1$.
\end{prop}

\begin{proof}
By Lemma \ref{Lemma-unital-1}, for any $k$ we have for the strong operator topology
$$
M_{\varphi_{\alpha}}(1_{\M_{n_{\alpha}}})
= \Psi_{\alpha} M_\varphi\Phi_{\alpha}(1_{\M_{n_{\alpha}}}) 
\xra[\alpha]{}
\Psi_{\alpha} M_\varphi(\Id_{\B(\L^2(X))})
=1_{\M_{n_{\alpha}}}.
$$
Recall that by Lemma \ref{Lemma-approx} the operator $M_{\varphi_{\alpha}} \co \M_{n_{\alpha}} \to \M_{n_{\alpha}}$ is the Schur multiplier on the matrix algebra $\M_{n_{\alpha}}$ associated with the matrix $\big[\frac{1}{\mu(A_{i})}\frac{1}{\mu(A_{j})}\int_{A_{i} \times A_{j}} \varphi \big]_{1 \leq i,j \leq n_{\alpha}}$. Now, it is clear that for any $1 \leq i \leq n_{\alpha}$ we have $\frac{1}{\mu(A_{i})^2} \int_{A_{i} \times A_{i}} \varphi \xra[\alpha]{} 1$. By the convergence martingale theorem \cite[Theorem 3.32]{HvNVW1}, we have the convergence
$$
\sum_{i,j=1}^{n_\alpha} \frac{1}{\mu(A_{i})} \frac{1}{\mu(A_{j})}\bigg(\int_{A_{i} \times A_{j}} \varphi \bigg) 1_{A_{i} \times A_{j}}
\ov{\eqref{Def-phi-alpha}}{=} \E_{\alpha}(\varphi) \to \varphi
$$ 
almost everywhere on $X \times X$. Note that if $x \in X_0$ we have (as \cite[(2.7) page 233]{Bri2})
$$
\E_{\alpha}(\varphi)(x,x)
=\frac{1}{\mu(A_{i_x})^2} \int_{A_{i_x} \times A_{i_x}} \varphi.
$$ 
The proof is complete.
\end{proof}

\subsection{Lusin filtration}
\begin{remark} \normalfont
\label{Remark-Lusin}
Let $(X,\cal{A},\mu)$ be a $\sigma$-finite measure space We could use the notion of Lusin $\mu$-filtration of \cite[Definition 2.2]{Jef1}, a filtration $(\mathcal{A}_k)_{k \in \mathbb{N}}$ for which there exists a conegligible subset $X_0$ and an increasing\footnote{\thefootnote. A partition $\alpha$ is finer than a partition $\beta$ if each element of $\alpha$ is a subset of some element of $\beta$.} sequence $(\alpha_k)$ of countable partitions of $X_0$ such that 
\begin{flalign}
& \label{isonormal-almost} \text{$\mathcal{A}_k$ is the $\sigma$-algebra generated by $\alpha_k$} \\
&\label{difference-1} \text{$\mathcal{A} \cap X_0=\vee_k \mathcal{A}_k$,} \\
&\label{difference-2} \text{for any $x \in X_0$ we have $0<\mu(U_k(x))<\infty$}.  
\end{flalign}
where $U_k(x)$ is the unique set of the partition $\alpha_k$ of $X$ which contains $x$.
Note that by \cite[Theorem 2.3]{Jef1} any $\sigma$-finite Borel measure $\mu$ on a Souslin space admits a Lusin $\mu$-filtration. See \cite[page 232]{Bri2} for related things.
\end{remark}

\subsection{Autres}

Let us recall the version of Mercer's theorem for $\sigma$-finite measure spaces stated in \cite[Theorem 31]{Car1}. Let $X$ be a $\sigma$-finite measure space. Let $\varphi \co X \times X \to \mathbb{C}$ be a measurable positive definite kernel such that $\int_X \varphi(x,x) \d x <\infty$. Then the map $\L^2(X) \to \L^2(X)$, $\xi \mapsto \int_{X} \varphi(\cdot,y)\xi(y) \d y$ is a well-defined trace-class positive operator. The eigenfunctions $f_n \in \L^2(X)$ of this operator associated with those $n$ such that $\lambda_n \not=0$ and normalized by $\norm{f_n}_2 = 1$ satisfy
\begin{equation}
\label{Mercer-Cartier}
\varphi(x,y) 
=\sum_{n=0}^{\infty} \lambda_n f_n(x) \ovl{f_n(y)}
\end{equation}
almost everywhere. Moreover, for any $x,y \in X$ the series $\sum_{n \geq 0} \lambda_n f_n(x) \ovl{f_n(y)}$ converges absolutely. Note that the assumption $\int_X \varphi(x,x) \d x$ of this theorem is satisfied if $X$ is finite and if $\varphi$ is bounded.

https://mathoverflow.net/questions/253678/the-topology-of-pointwise-convergence-with-the-adjoint-operator-on-a-von-neumann
https://mathoverflow.net/questions/258829/operator-topologies-on-l-inftyx-mu

\vspace{0.2cm}

-preuve Lemma \ref{Lemma-iso-1}.

-th avec les angles Noncommutative harmonic analysis on semigroups. Yong Jiao,

-resultats ergodic consequence de l'angle

-cas d'un mult de Schur et d'un mult de Fourier (i.e. sans semigroupe)

-preuve de la carac de Schur

-mult de Birman

-open the dorr to transference

-continuite des path

\subsection{Konig}

Let us state a \textit{corrected} version of Mercer's theorem for finite measure spaces stated in \cite[Theorem 3.a.1 page 145]{Kon1}. Note that the original version (and its proof) contains some problematic errors\footnote{\thefootnote. In particular, the estimate $\sup_n \norm{f_n}_\infty <\infty$ is note proved and seems incorrect and consequently the arguments of the convergence of  \cite[(3.3)]{Kon1} and of $\sum_{n \in \N} \lambda_n(T_k) (f,f_n) f_n$ are inexact.}. However, the statement can be corrected as follows using the \textit{same} nice ideas. The only contribution is to provide a correct statement and proof.

\begin{thm}
\label{Mercer-corrected}
Let $X$ be a finite measure space and $\varphi \in \L^\infty(X \times X)$ such that the Hilbert-Schmidt operator $K_\varphi \co \L^2(X) \to \L^2(X)$ is positive. 
The eigenfunctions $f_n \in \L^2(X)$ of the operator $K_{\varphi}$ associated with those $n$ such that $\lambda_n \not=0$, and normalized by $\norm{f_n}_2 = 1$, actually belong to $\L^\infty(X)$ with  
\begin{equation}
\label{Mercer}
\sup_{n \geq 0} \lambda_n \norm{f_n}_\infty^2 < \infty
\quad \text{and} \quad
\varphi(x,y) 
=\sum_{n=0}^{\infty} \lambda_n f_n(x) \ovl{f_n(y)}
\end{equation}
holds almost everywhere where the series converges absolutely almost everywhere and in $\L^\infty(X \times X)$. 
\end{thm}

\begin{proof}
The positive map $K_\varphi \co \L^2(X) \to \L^2(X)$ has a unique positive square root $S \co \L^2(X) \to \L^2(X)$. For any $\xi \in \L^1(X)$ we have
\begin{align*}
\MoveEqLeft
\norm{S(\xi)}^2_2
=\la S(\xi), S(\xi)\ra_{\L^2(X)}            
=\big\la S^2(\xi), \xi \big\ra_{\L^2(X)} 
=\la K_\varphi(\xi), \xi\ra_{\L^2(X)} \\
&=\iint_{X \times X} \varphi(x,y) \ovl{\xi(x)} \xi(y) \d x \d y
\leq \norm{\varphi}_{\L^\infty(X \times X)} \norm{\xi}_{\L^1(X)}^2.
\end{align*}
This proves that we have a well-defined continuous operator $S \co \L^1(X) \to \L^2(X)$. By duality, we have a bounded  operator $S \co \L^2(X) \to \L^\infty(X)$ with $
\norm{S}_{\L^2(X) \to \L^\infty(X)} 
\leq \norm{\varphi}_{\L^\infty(X \times X)}^{\frac{1}{2}}$. For any integer $n$, note that
$$
\lambda_n \norm{f_n}_\infty^2
=\bnorm{\sqrt{\lambda_n}f_n}_\infty^2
=\norm{S(f_n)}_\infty^2 
\leq \norm{S}_{\L^2(X) \to \L^\infty(X)}^2. 
$$
Hence, we deduce the estimate of \eqref{Mercer}. Recall that 
the embedding $i \co \L^\infty(X) \mapsto \L^2(X)$ is 2-summing with $\norm{i}_{\pi_2, \L^\infty(X) \mapsto \L^2(X)} =\mu(X)^{\frac{1}{2}}$. Consequently, the operator $S \co \L^2(X) \to \L^2(X)$ is 2-summing by composition with 
\begin{equation}
\label{estimate-pi_2}
\norm{S}_{\pi_2, \L^2(X) \to \L^2(X)} 
\leq \mu(X)^{\frac{1}{2}} \norm{\varphi}_{\L^\infty(X \times X)}^{\frac{1}{2}}
\end{equation}
Now we have using \cite[Proposition 2.a.1 page 79]{Kon1} in the first inequality
\begin{align*}
\MoveEqLeft
\sum_{n=0}^\infty \lambda_n
=\sum _{n=0}^\infty \lambda_n(S)^2            
\leq \norm{S}_{\pi_2, \L^2(X) \to \L^2(X)} 
\ov{\eqref{estimate-pi_2}}{=} \mu(X) \norm{\varphi}_{\L^\infty(X \times X)}.
\end{align*}
Hence $K_\varphi$ is trace-class. Note that $K_\varphi(\xi)=\sum_n \lambda_n\la \xi, f_n\ra f_n$ for any $\xi \in \L^2(X)$. This series converges absolutely almost everywhere. The end of the proof is classical. 
\end{proof}

See the introduction of \cite{StS1} for a nice discussion of different versions of Mercer's theorem. Note also the version of Mercer's theorem stated \cite[Theorem 31]{Car1}.

\begin{proof}
Now, we prove to the second implication. Suppose that $\varphi$ is an integrally positive definite kernel. For any $\xi \in \L^2(X)$, we have the inequality \eqref{Cartier-int-pos}, i.e. $\la K_{\varphi}(\xi), \xi \ra_{\L^2(X)} \geq 0$. We infer that the operator $K_\varphi$ is positive. 
Using the above version of Mercer's theorem, we can define for any integer $n \geq 0$ the element $\alpha_n \ov{\mathrm{def}}{=} \sqrt{\lambda_n}f_n$ of $\L^2(X)$. For almost all $x \in X$, we have
$$
\sum_{n=0}^{\infty} |\alpha_n(x)|^2
=\sum_{n=0}^{\infty} \lambda_n|f_n(x)|^2
\leq \sum_{n=0}^{\infty} \lambda_n \norm{f_n}^2_\infty
<\infty.
$$
So we have an almost everywhere defined function $\alpha \co X \to \ell^2$, $x \mapsto (\alpha_n(x))_{n \geq 0}$. From \eqref{Mercer}, it is clear that 
$$
\varphi(x,y)
=\sum_{n=0}^{\infty}  \sqrt{\lambda_n}f_n(x)\ovl{\sqrt{\lambda_n}f_n(y)}
=\la \alpha(x),\alpha(y) \ra_{\ell^2}
$$ 
for almost all $x,y \in X$. Extending $\alpha$ on the whole $X$, we conclude with Example \ref{Example-kernel-2}. The remaining assertions are easy and left to the reader.
\end{proof}

\begin{proof}
The Hilbert-Schmidt operator $K_\varphi \co \L^2(X) \to \L^2(X)$ is compact and positive. By the spectral theorem, we can write
\begin{equation}
\label{Spectral-decomposition}
K_\varphi(\xi)
=\sum_{i=0}^\infty \lambda_i \la u_i,\xi \ra u_i, \quad \xi \in \L^2(X)
\end{equation}
with $\lambda_i \geq 0$ and $u_i \in \L^2(X)$. For any $i \geq 0$, define $\alpha_i =\sqrt{\lambda_i}u_i$ and the element $\alpha \co X \to \ell^2$, $x \mapsto (\alpha_i(x))$ in \textbf{?$\ell^2$?}. For any $\xi \in \L^2(X)$ and any $y$, we have  
\begin{align*}
\MoveEqLeft
\int_X \la \alpha(x), \alpha(y) \ra_{\ell^2} \xi(x) \d x
=\int_X \sum_{i=0}^{\infty} \ovl{\alpha_i(x)}\alpha_i(y) \xi(x) \d x
=\int_X \sum_{i=0}^{\infty} \lambda_i \ovl{u_i(x)}u_i(y)\xi(x) \d x  \\         
&?=?\sum_{i=0}^{\infty} \lambda_i u_i(y) \int_X \ovl{u_i(x)} \xi(x) \d x 
=\sum_{i=0}^\infty \lambda_i \la u_i,\xi \ra u_i(y) 
\ov{\eqref{Spectral-decomposition}}{=} (K_\varphi \xi)(y)
\ov{\eqref{Def-de-Kf}}{=} \int_{X} \varphi(x,y) \xi(y) \d x.
\end{align*}
Hence $
\la \alpha(x),\alpha(y) \ra_{\ell^2}
=\varphi(x,y)
$ for almost all $x,y \in X$. We conclude with Example \ref{Example-kernel-2}.
\end{proof}

\begin{remark} \normalfont
\label{remark-open-question}
It would be very useful to know if we can remove the assumption ``which induces a bounded Schur multiplier'' in the above result.
\end{remark}

The following is a partial generalization of \cite[Proposition 5.4]{Arh1}. 

	%
	%
%

\begin{proof}
Observe that $\varphi$ belongs to $\L^2(X \times X)$. So we have a well-defined Hilbert-Schmidt operator $K_\varphi \co \L^2(X) \to \L^2(X)$. If $\xi$ belongs to the Hilbert space $\L^2(X)$, note that $\ovl{\xi} \ot \xi$ belongs to $\L^2(X \times X)$. Hence the function $(x,y) \mapsto \varphi(x,y) \ovl{\xi(x)} \xi(y)$ is integrable on $X \times X$ and by Fubini's theorem we have 
\begin{equation}
\label{Fubini}
\la K_{\varphi}(\xi), \xi \ra_{\L^2(X)}
\ov{\eqref{Def-de-Kf}}{=} \iint_{X \times X} \varphi(x,y) \ovl{\xi(x)} \xi(y) \d x \d y .
\end{equation} 
We conclude with the above Cartier's observation.
\end{proof}

A generalization of the above result to $\sigma$-finite measures spaces would also be welcome, maybe using an elementary argument to reduce the $\sigma$-finite case to the finite case. It would be interesting to know if the product of two integrally positive
definite kernels of $\L^\infty(X)$ is again an integrally positive definite kernel.

\begin{remark} \normalfont
\label{Rem-int-pos-continuous}
Let $X$ be a locally compact topological space and $\mu$ be a Radon measure on $X$ with support $X$. Consider a \textit{continuous} function $\varphi \co X \times X \to \mathbb{C}$. Then the author believe that it is folklore that $\varphi$ is an integrally positive definite kernel if and only if $\varphi$ is a positive definite kernel on $X \times X$.
\end{remark}

\end{document}